\documentclass[12pt,a4paper, leqno]{article}
\usepackage{amssymb, amscd, amsthm, amsmath,amstext}
\usepackage[all]{xy}

\usepackage[colorlinks, linkcolor=blue, anchorcolor=green, citecolor=blue]{hyperref}

\newtheorem{thm}{Theorem}[section]

\newtheorem{lemma}[thm]{Lemma}

\theoremstyle{remark}

\theoremstyle{definition}

\numberwithin{equation}{thm}
\newcommand{\newpara}{\noindent\refstepcounter{thm}{\bf(\thethm)\;}}

\newcommand{\N}{\mathbb N}

\newcommand{\Z}{\mathbb Z}

\newcommand{\al}{\alpha}

\newcommand{\car}{\mathrm{char}}

\newcommand{\ind}{\mathrm{ind}}
\newcommand{\per}{\mathrm{per}}
\newcommand{\sd}{\mathrm{sd}}
\newcommand{\cd}{\mathrm{cd}}
\newcommand{\Br}{\mathrm{Br}}
\newcommand{\Nrd}{\mathrm{Nrd}}
\newcommand{\lra}{\longrightarrow}
\newcommand{\ov}[1]{\overline{#1}}
\newcommand{\simto}{\xrightarrow{\sim}}

\begin{document}
\title{\textbf{Reduced norms of division algebras over complete discrete valuation fields of local-global type}}
\author{Yong HU\footnote{supported by a grant from the National Natural Science Foundation of China (Project No. 11801260).}}

\maketitle

\begin{abstract}
Let $F$ be a complete discrete valuation field whose residue field $k$ is a global field of positive characteristic $p$. Let $D$ be a central division $F$-algebra of $p$-power degree. We prove that the subgroup of $F^*$ consisting of reduced norms of $D$ is exactly the kernel of the cup product map $\lambda\in F^*\mapsto (D)\cup(\lambda)\in  H^3(F,\,\mathbb{Q}_{p}/\Z_{p}(2))$, if either $D$ is tamely ramified or of period $p$. This gives a $p$-torsion counterpart of a recent theorem of Parimala, Preeti and Surech, where the same result is proved  for division algebras of prime-to-$p$ degree.
\end{abstract}




\medskip



\section{Introduction}

 Let $F$ be  a field and let $D$ be a central division algebra of degree $n$ over $F$. First assume $n$ is not divisible by the characteristic of $F$. Using Kummer theory we may consider the Brauer class $(D)$ of $D$ as an element of the Galois cohomomology group $H^2(F,\,\Z/n(1)):=H^2(F,\,\mu_n)$; and the isomorphism $F^*/F^{*n}\cong H^1(F,\,\mu_n)$ gives rise to a natural map $F^*\to H^1(F,\,\mu_n):\;\lambda\mapsto (\lambda)$. The cup product map
 \[
 F^*\lra H^3(F,\,\Z/n(2)):=H^3(F,\,\mu_n^{\otimes 2})\;;\quad\lambda\longmapsto (D)\cup (\lambda)\,
 \]vanishes on the group $\Nrd(D^*)$ of (nonzero) reduced norms of $D$, by a norm principle for reduced norms (\cite[Prop.\;2.6.8]{GilleSzamuely17}) and the projection formula in Galois cohomology. So we have an induced homomorphism
 \[
\mathcal{R}_D\,:\;\; F^*/\Nrd(D^*)\lra H^3(F,\,\Z/n(2))\;;\quad\lambda\longmapsto (D)\cup (\lambda)\,,
\]which we call the \emph{Rost invariant} of $D$. (This is the Rost invariant of the semisimple simply connected group $\mathbf{SL}_1(D)$. See e.g. \cite{Merkurjev03inGMScohinv}.)

By well known theorems of Merkurjev and Suslin (\cite[Theorem\;12.2]{MS82NormResidue}, \cite[Theorems\;24.4 and 24.8]{Suslin85}), if $n=\deg(D)$ is square-free or the cohomological $\ell$-dimension of $F$ is $\le 2$ for all prime divisors $\ell$ of $n$, then the Rost invariant $\mathcal{R}_D$ is injective. A recent work of Parimala, Preeti and Suresh \cite{ParimalaPreetiSuresh2018} proves this Rost injectivity in a remarkable new case: If $F$ is a one-variable function field over a (non-archimedean) local field $K$ and if $n=\deg(D)$ is coprime to the residue characteristic of $K$, then $\mathcal{R}_D$ is injective. Their proof relies heavily on patching techniques developed by Harbater, Hartmann and Krashen (\cite{HHK}, \cite{HH}, \cite{HHK15AJM}, \cite{HHK15IMRN}, etc.), and it has been an important intermediate step to analyze the situation over the completion of $F$ at a divisorial valuation $v$. The residue field of such a (discrete) valuation is a global field of positive characteristic $p$ or a local field with residue characteristic $p$. For division algebras of prime-to-$p$ degree, the Rost injectivity over the completion $F_v$ was proved in \cite[Thm.\;4.12]{ParimalaPreetiSuresh2018}.

We note that if $F$ is a 2-local field, i.e., a complete discrete valuation field whose residue field $k$ is a (non-archimedean) local field (with finite residue field) in the usual sense, then the Rost invariant $\mathcal{R}_D$ is injective for all division algebras $D$ over $F$ (including the case where $\deg(D)$ is divisible by the residue characteristic of $k$). This was a consequence of Kato's work on two-dimensional local class field theory (\cite[p.657, $\S$3.1, Thm.\;3 (2)]{KatoII80}). So, it is natural to ask whether the same is true when the residue field $k$ is global.

In this paper we prove the following:

\begin{thm}\label{1p1}
  Let $F$ be a complete discrete valuation field whose residue field $k$ is a global field of characteristic $p>0$. Let $D$ be a central division algebra of $p$-power degree over $F$.

  Then the Rost invariant $\mathcal{R}_D$ is injective in each of the following cases:

  \begin{enumerate}
    \item $D$ is tamely ramified, i.e., $D$ splits over the maximal unramified extension of $F$;
    \item the period of $D$ is $p$.
  \end{enumerate}
\end{thm}

After some reviews of preliminary facts in $\S$\ref{sec2}, we will prove the theorem in $\S$\ref{sec3} (see \eqref{3p6}) admitting some technical results whose proofs occupy the whole $\S$\ref{sec4}.

\medskip

A few remarks about the theorem are worth mentioning at this moment.

First, in the theorem the field $F$ can have characteristic 0 or $p$. In fact, for any field $K$ of characteristic $p$ and any $p$-power $n$, the groups $H^{r+1}(K,\,\Z/n(r)),\,r\in\N$ can be defined in a suitable way so that they have almost all the properties the Galois cohomology groups $H^{r+1}(K,\,\mu_n^{\otimes r})$ have in the case $p\nmid n$ (or at least, all the properties we need in this paper remain valid as in the case $p\nmid n$). In $\S$\ref{sec2} we will recall some most useful facts in this respect and we refer the reader to \cite[$\S$3.2]{KatoII80} and \cite[Appendix A]{GMS} for more details. In particular, $H^2(K,\,\Z/n(1))$ can be identified with the $n$-torsion subgroup $\Br(K)[n]$ of the Brauer group of $K$ and  ``cup product'' maps
\[
H^{r+1}(K,\,\Z/n(r))\times (K^*/K^{*n})\longrightarrow H^{r+1}(K,\,\Z/n(r+1))
\]can be defined and have similar properties. So the Rost invariant map $\mathcal{R}_D$ in Theorem\;\ref{1p1} can be defined in the same way as before.

Secondly, for the field $F$ in Theorem\;\ref{1p1}, if $D$ has period $p$ and is not tamely ramified, then by \cite[p.337, $\S$3, Lemma\;5]{KatoI79}, the degree of $D$ must be $p$. Therefore, the result in this case follows directly form the Merkurjev--Suslin theorem \cite[Thm.\;12.2]{MS82NormResidue} or its $p$-primary counterpart \cite[p.94, Thm.\;6]{Gille00}. So we need only to consider the first case of the theorem.

Still another remark is that in the theorem we may have only assumed $F$ henselian and excellent instead of the completeness assumption. In fact, letting $\hat{F}$ denote the completion of $F$ and $\hat{D}=D\otimes_F\hat{F}$, the natural map
\[
H^1(F,\,\mathbf{SL}_1(D))=F^*/\Nrd(D^*)\lra H^1(\hat{F}\,,\,\mathbf{SL}_1(D))=\hat{F}^*/\Nrd(\hat{D}^*)
\]is injective by Greenberg's theorem \cite{Greenberg66}. Therefore, the injectivity of $\mathcal{R}_{\hat{D}}$ implies that of $\mathcal{R}_D$.

\medskip

We expect that the theorem remain true without any restriction on $D$ and that the same result can be shown for function fields of curves over local fields, thereby extending \cite[Thm.\;1.1]{ParimalaPreetiSuresh2018} to the case of $p$-power degree algebras.

\section{Kato--Milne cohomology and Brauer groups in characteristic $p$}\label{sec2}

Fix a prime number $p$ and a field $k$ of characteristic $p$.

We need to use the Kato--Milne cohomology groups $H^{r+1}(k,\,\Z/p^m(r)),\,r\in\N,\,m\in\N^*$, which can be viewed as the $p$-primary counterpart of the Galois cohomology groups $H^{r+1}(k,\,\mu_n^{\otimes r}),\,r\in\N$ with $n$ coprime to $p$. It does not seem to be absolutely necessary to give the precise definition oof the Kato--Milne cohomology here. The interested readers are referred to \cite{KatoII80} and \cite{Milne76AnnSciENS}, or for another approach, \cite[Appendix\;A]{GMS}. However, we do need to know some basic facts about the groups $H^{r+1}(k,\,\Z/p^m(r))$ which we shall now recall.

\

\newpara\label{2p1} Let $k$ be a field of characteristic $p$. For integers $r\ge 0$ and $m\ge 1$, the group $H^{r+1}(k,\,\Z/p^m(r))$ can be defined as in \cite[$\S$3.2]{KatoII80} or \cite[Appendix\;A]{GMS}.. The following statements hold:

\begin{enumerate}
  \item $H^1(k,\,\Z/p^m(0))$ coincides with the usual Galois cohomology group $H^1(k,\,\Z/p^m)$, with the Galois action on $\Z/p^m$ being trivial.
  \item $H^2(k,\,\Z/p^m(1))$ can be identified with the $p^m$-torsion subgroup $\Br(k)[p^m]$ of the Brauer group of $k$.
  \item We write
  \[
  \begin{split}
  H^{r+1}_{p^m}(k)&=H^{r+1}(k\,,\,\Z/p^m(r))\;,\\
  \text{and}\quad H^{r+1}(k)&=\varinjlim_{m}H^{r+1}_{p^m}(k)=\varinjlim_{m}H^{r+1}(k\,,\,\Z/p^m(r))\;.
  \end{split}
  \]For every $m\in\N^*$, the natural map $H^{r+1}_{p^m}(k)\to H^{r+1}(k)$ is injective and its image is the subgroup of $p^m$-torsion elements in $H^{r+1}(k)$.
  \item For every $q\in\N$, let $K_q$ be the functor that associates to each field its $q$-th Milnor $K$-group. Then there are natural pairings
  \[
  H^{r+1}_{p^m}(k)\times \left( K_q(k)/p^m\right)\lra H^{r+q+1}_{p^m}(k)\,,
  \]which we call \emph{cup products} and denote by $(\al,\,\lambda)\mapsto \al\cup \lambda$, by analogy to the prime-to-$p$ case. (In this paper, for any abelian group $A$ and any integer $n$, we write $A/n=A\otimes_{\Z}(\Z/n\Z)$.)
  \item If $L/k$ is a field extension, there are restriction maps
  \[
  \mathrm{Res}_{L/k}\;:\;\; H^{r+1}_{p^m}(k)\lra H^{r+1}_{p^m}(L)\;;\quad \theta\longmapsto \theta_L\,.
  \]
  If $L/k$ is a finite extension, we have corestriction maps
  \[
  \mathrm{Cor}_{L/k}\;:\;\; H^{r+1}_{p^m}(L)\lra H^{r+1}_{p^m}(k)\;.
  \]As in the prime-to-$p$ case, $\mathrm{Cor}_{L/k}\circ\mathrm{Res}_{L/k}$ is the multiplication by $[L:k]$. Moreover, the usual projection formulas hold:
  \[
  \begin{split}
    \mathrm{Cor}_{L/k}(\theta_L\cup \mu)&=\theta\cup N_{L/k}(\mu)\,,\quad\forall\;\theta\in H^{r+1}_{p^m}(k)\,,\;\forall\;\mu\in K_q(L)/p^m\,,\\
      \mathrm{Cor}_{L/k}(\theta'\cup \lambda_L)&=\mathrm{Cor}_{L/k}(\theta')\cup \lambda\,,\quad\forall\;\theta'\in H^{r+1}_{p^m}(L)\,,\;\forall\;\lambda\in K_q(k)/p^m\,.
  \end{split}
  \]Here for Milnor $K$-groups we denote by $N_{L/k}: K_q(L)\to K_q(k)$ the norm homomorphism.
\end{enumerate}

\newpara\label{2p2}
Now let $F$ be a henselian excellent discrete valuation field with residue field $k$ (of characteristic $p$). ($F$ may have characteristic $0$ or $p$.) For all $r\in\N$ and $m\in\N^*$ we define
\begin{equation}\label{2p2p1}
  H^{r+1}_{p^m}(F)_{tr}:=\ker\big(\mathrm{Res}\,:\; H^{r+1}_{p^m}(F)\lra H^{r+1}_{p^m}(F_{nr})\big)\,,
\end{equation}where $F_{nr}$ denotes the maximal unramified extension of $F$. There is a natural inflation map (cf. \cite[p.659, $\S$3.2, Definition\;2]{KatoII80})
\[
  \mathrm{Inf}\;:\;\; H^{r+1}_{p^m}(k)\lra H^{r+1}_{p^m}(F)
\]and the choice of a uniformizer $\pi\in F$ defines a homomorphism (cf. \cite[p.659, $\S$3.2, Lemma\;3]{KatoII80})
\[
h_{\pi}\,:\;\; H^r(k)\lra H^{r+1}(F)\;;\quad w\longmapsto \mathrm{Inf}(w)\cup (\pi)\,.
\](By convention, $H^r_{p^m}=0$ if $r=0$.) The images of $\mathrm{Inf}$ and $h_{\pi}$ are both contained in $H^{r+1}_p(F)_{tr}$. It is proved in \cite[p.219, Thm.\;3]{Kato82} that the above two maps induces an isomorphism
\begin{equation}\label{2p2p2}
  \mathrm{Inf}\oplus h_{\pi}\;:\;\; H^{r+1}_{p^m}(k)\oplus H^{r}_{p^m}(k)\simto H^{r+1}_p(F)_{tr}\;.
\end{equation}
We can thus define a \emph{residue map} $\partial:H^{r+1}_p(F)_{tr}\to H^r_p(k)$, which fits into the split exact sequence
\begin{equation}\label{2p2p3}
  0\lra H^{r+1}_{p^m}(k)\overset{\mathrm{Inf}}{\lra}H^{r+1}_{p^m}(F)_{tr}\overset{\partial}{\lra} H^r_{p^m}(k)\lra 0\,.
\end{equation}In this paper we are mostly interested in the residue maps defined on $H^2$ and $H^3$. We have the following formulae:
\begin{equation}\label{2p2p4}
  \partial(\mathrm{Inf}(\chi_0)\cup(\lambda))=v_F(\lambda).\chi_0\;\in\;H^1(k)\,,
\end{equation}
and
\begin{equation}\label{2p2p5}
\partial(\mathrm{Inf}(\chi_0)\cup(\lambda)\cup(\mu))=\chi_0\cup (-1)^{v(\lambda)v(\mu)}\ov{\lambda^{v(\mu)}\mu^{-v(\lambda)}}\;\in\;H^2(k)\,,
\end{equation}
for all  $\lambda,\,\mu\in F^*$ and $\chi_0\in H^1(k)$.

In terms of Brauer groups, the case $r=1$ of \eqref{2p2p3} can be interpreted as the exact sequence
\begin{equation}\label{2p2p6}
  0\lra \Br(k)[p^m]\overset{\mathrm{Inf}}{\lra} \Br_{tr}(F)[p^m]\overset{\partial}{\lra} H^1(k,\,\Z/p^m)\lra 0\,,
\end{equation}where
\[
\Br_{tr}(F):=\ker(\Br(F)\lra \Br(F_{nr}))\,
\]is called the \emph{tame} or \emph{tamely ramified}  part  of $\Br(F)$.

\

\newpara\label{2p3} For a field $k$ of characteristic $p$, there are two useful variants of the cohomological $p$-dimension $\cd_p(k)$: the \emph{separable $p$-dimension} $\sd_p(k)$ (\cite[p.62]{Gille00}) and \emph{Kato's $p$-dimension} $\dim_p(k)$ (\cite[p.220]{Kato82}). They are defined as follows:
\[
\begin{split}
\sd_p(k)&:=\inf\{r\in\N\,|\,H^{r+1}_p(k')=0\,, \;\forall\; \text{finite separable extension } k'/k\}\,,\\
\dim_p(k)&:=\inf\{r\in\N\,|\,[k:k^p]\le p^r\text{ and } H^{r+1}_p(k')=0\,, \;\forall\; \text{finite extension } k'/k\}\,.
\end{split}
\]
It is easy to see that $\sd_p(k)\le \dim_p(k)$. Moreover, it can be shown that
\[
\log_p[k:k^p]\le \dim_p(k)\le \log_p[k:k^p]+1\,.
\] Therefore, we have the following implications:
\[
\dim_p(k)\le 1 \Longrightarrow [k:k^p]\le p \Longrightarrow \dim_p(k)\le 2\Longrightarrow \sd_p(k)\le 2\,.
\]
Notice also that the condition $\sd_p(k)\le 1$ is equivalent to saying that $\Br(L)$ has no $p$-torsion for all algebraic extensions $L/k$ (cf. \cite[Chap.\;II, $\S$3.1]{SerCohGal94}).

If $k$ is a global function field of characteristic $p$, then $[k:k^p]=p$ and $\dim_p(k)=\sd_p(k)=2$.

For fields of characteristic different from $p$, both the separable $p$-dimension and Kato's $p$-dimension are defined to be the same as the cohomological $p$-dimension.

\section{Proof of main theorem}\label{sec3}

We now proceed to the proof of Theorem\;\ref{1p1}. We basically follow the same strategy as  Parimala, Preeti and Suresh's paper \cite{ParimalaPreetiSuresh2018}. Our main contribution is the extension of some key lemmas in \cite{ParimalaPreetiSuresh2018} to the characteristic $p$ case. At some points the proofs of our generalized versions involve some subtleties and have to be treated with special care.

\medskip

\newpara\label{3p1} Let us fix the following notation for the whole section:

\begin{enumerate}
  \item For each $r\in\N$ we write $H^{r+1}(\cdot)=\varinjlim_{m}H^{r+1}(\cdot\,,\,\Z/p^m(r))$ as in $\S$\ref{sec2}.

  Here $H^{r+1}(\cdot,\,\Z/p^m(r))$ is the Kato--Milne cohomology in characteristic $p$ and is the Galois cohomology $H^{r+1}(\cdot,\,\mu_{p^m}^{\otimes r})$ in characteristic different from $p$.
  \item Let $F$ be a complete discrete valuation field with residue field $k$.

  We assume that $\car(k)=p>0$.
  \item Let $v=v_F$ denote the normalized discrete valuation on $F$ and fix a uniformizer $\pi\in F$.
  \item If $L/F$ is a finite extension, we put
  \[
  \begin{split}
    v_L&=\text{the normalized discrete valuation on } L\,,\\
    U_L&=\{x\in L\,|\,v_L(x)=0\}\,,\\
    \text{and }\;\; U^n_L&=\{x\in L\,|\,v_L(x-1)\ge n\}\,,\quad\forall\; n\ge 1\,.
  \end{split}
  \]
  \item Let $0\neq\al\in \Br(F)[p^{\infty}]$ be a nonzero Brauer class of $p$-power index over $F$.

  We assume that $\al$ is tamely ramified, i.e., $\al\in\Br_{tr}(F)=\ker(\Br(F)\to\Br(F_{nr}))$.
  \item Let $\chi_0=\partial(\al)\in H^1(k)$ be the residue of $\al$, which is well defined according to the tameness assumption on $\al$ \eqref{2p2p3}.

The character $\chi_0 \in H^1(k)$ can be determined by a pair $(E_0/k,\,\ov{\sigma})$, where $E_0/k$ is the cyclic extension and $\ov{\sigma}$ is a  generator of the cyclic Galois group $\mathrm{Gal}(E_0/k)$. The correspondence between $\chi_0$ and $(E_0/k,\,\ov{\sigma})$ is established by requiring  that the continuous homomorphism $\chi_0:\mathrm{Gal}(k_s/k)\to \mathbb{Q}_{p}/\Z_{p}$ has kernel $\mathrm{Gal}(k_s/E_0)$, $k_s$ denoting a fixed separable closure of $k$, and that $\ov{\sigma}\in \mathrm{Gal}(E_0/k)$ is the generator which is mapped to the canonical generator of the cyclic group $\mathrm{Im}(\chi_0)$. Since the role played by $\ov{\sigma}$ is almost never explicit in our arguments, we will simply write $\chi_0=(E_0/k)$.

By the \emph{canonical lifting} $\chi\in H^1(F)$ of $\chi_0$ we shall mean the image of $\chi_0$ under the inflation map $H^1(k)\to H^1(F)$. Explicitly, $\chi$ is defined by the pair $(E/F,\,\sigma)$ where $E/F$ is the unramified extension with residue field extension $E_0/k$ and $\sigma\in\mathrm{Gal}(E/F)$ is the generator corresponding to $\ov{\sigma}$ via the natural isomorphism $\mathrm{Gal}(E/F)\simto \mathrm{Gal}(E_0/k)$. Just as for $\chi_0$, we will write $\chi=(E/F)$ for short. For any element $b\in F^*$, we write
\[
(E/F,\,b)=\chi\cup(b)\;\in\; H^2(F)=\Br(F)[p^{\infty}]\,.
\]
  \item As in \cite[Lemma\;4.1]{ParimalaPreetiSuresh2018}, we may write
\begin{equation}\label{3p1p1}
  \al=\al'+(E/F\,,\,\pi)\;\in\;\Br(F)\quad\text{ with }\;\al'\in \Br(F) \text{ unramified}\,.
\end{equation}
Here we call a Brauer class $\beta\in\Br(F)[p^{\infty}]$ \emph{unramified} if $\beta\in \Br_{tr}(F)$ and if $\partial(\beta)=0$. Equivalently, an unramified element of $\Br(F)[p^{\infty}]$ is an element in the image of the inflation map $\Br(k)[p^{\infty}]\to \Br_{tr}(F)[p^{\infty}]$ \eqref{2p2p6}.

We denote by $\bar\al'\in\Br(k)[p^{\infty}]$ the element that is mapped to  $\al'\in\Br(F)$ under the inflation map.
\item Let $\lambda\in F^*$ be such that $\al\cup(\lambda)=0\in H^3(F)$. We write
\begin{equation}\label{3p1p2}
  \lambda=\theta.(-\pi)^s\quad \text{with }\; s=v_F(\lambda)\in\Z \text{ and } \theta \in U_F\,.
\end{equation}Computing the residue of $\al\cup(\lambda)$ we see that
\begin{equation}\label{3p1p3}
s\al'=(E/F\,,\,\theta)\quad\text{and}\quad s\al=(E/F\,,\,(-1)^s\lambda)\;
\end{equation}(cf. \cite[Lemma\;4.7]{ParimalaPreetiSuresh2018}).
\end{enumerate}

To prove our main theorem, we need to show $\lambda$ is a reduced norm for $\al$ when $k$ is a global function field.

\

\newpara\label{3p2} With notation as above, by an \emph{inductive pair} for $(\al,\,\lambda)$ we mean a pair $(L,\,\mu)$ consisting of a separable field extension $L/F$ of degree $p$ and an element $\mu\in L^*$ such that $N_{L/F}(\mu)=\lambda$, the index $\ind(\al_L)$ of $\al_L$ is strictly smaller than $\ind(\al)$, and $\al_L\cup(\mu)=0\in H^3(L)$.

By the norm principle for reduced norms, we may use induction on the index of $\al$ to conclude that $\lambda$ is a reduced norm of $\al$, as soon as an inductive pair exists.

\begin{lemma}\label{3p3}
  With notation and hypotheses as above, suppose that $s=v_F(\lambda)$ is coprime to $p$. Let $L=F(\sqrt[p]{-\lambda})$ and $\mu=-\sqrt[p]{-\lambda}$.

  Then $\ind(\al_L)<\ind(\al)$ and $\partial(\al_L\cup(\mu))=0$. Therefore, if $\sd_p(k)\le 2$ (e.g., $k$ is a global function field), then $(L,\,\mu)$ is an inductive pair for $(\al,\,\lambda)$.
\end{lemma}
\begin{proof}
  The proof can be done in the same manner as in the prime-to-$p$ case, for which the reader is referred to \cite[Lemma\;4.7]{ParimalaPreetiSuresh2018}.
\end{proof}

Thanks to the above lemma, we may assume $p\,|\,s$. In the rest of this section, we write $s=rp$, so that
\[
\lambda=\theta(-\pi)^{rp}\quad\text{and}\quad  rp\al'=(E/F,\,\theta)
\]according to \eqref{3p1p2} and \eqref{3p1p3}.

\begin{lemma}\label{3p4}
  With notation and hypotheses as above (in particular $s=v_F(\lambda)=rp$), suppose that $L_0/k$ is a separable field extension of degree $p$ and that $\xi_0\in L_0^*$ satisfies
  \[
  N_{L_0/k}(\xi_0)=\bar\theta\quad\text{and}\quad r\bar\al'_{L_0}=(E_0L_0/L_0,\,\xi_0)\,.
  \]Let $L/F$ be the unramified extension with residue field $L_0/k$.

  Then there exists an element $\xi\in U_L$ such that
  \begin{enumerate}
    \item $\bar{\xi}=\xi_0$ in $L_0$;
    \item $N_{L/F}(\xi)=\theta$;
    \item $\partial(\al_L\cup(\xi.(-\pi)^r))=0$. (Thus, if $\sd_p(k)\le 2$, we have $\al_L\cup(\xi.(-\pi)^r)=0$.)
  \end{enumerate}
\end{lemma}
\begin{proof}
  A direct calculation (using \eqref{2p2p5}) shows
  \[
  \begin{split}
    \partial_L(\al_L\cup(\xi.(-\pi)^r))&=r\bar\al'_{L_0}+\partial_L((E/F,\,\pi)_L\cup (\xi.(-\pi)^r))\\
    &=r\bar\al'_{L_0}-(E_0L_0/L_0,\,\bar\xi)\\
    &=(E_0L_0/L_0\,,\,\xi_0)-(E_0L_0/L_0,\,\bar\xi)\;.
  \end{split}
  \]So condition (3) is implied by condition (1). We need only to find a unit $\xi\in L$ satisfying (1) and (2).

  Since $L/F$ is an unramified extension, we have a commutative diagram with exact rows
  \begin{equation}\label{3p4p1}
    \begin{CD}
      0 @>>> U^1_L @>>> U_L @>>> L_0^* @>>> 0\\
      && @V{N_{L/F}}VV @VV{N_{L/F}}V @VV{N_{L_0/k}}V\\
        0 @>>> U^1_F @>>> U_F @>>> k^* @>>> 0
    \end{CD}
  \end{equation}
  and $N_{L/F}(U^1_L)=U^1_F$ by \cite[$\S$V.2, Prop.\;3]{SerreLocalFields}. In particular, the natural map $U_F\to k^*;\,x\mapsto \bar x$ induces an isomorphism
  \[
  w\,:\;\;U_F/N_{L/F}(U_L)\simto k^*/N_{L_0/k}(L_0^*)\,,
  \]and the induced map
  \[
  w'\;:\;\;\ker(N_{L/F}: U_L\to U_F)\lra \ker(N_{L_0/k}: L_0^*\lra k^*)
  \]is surjective. By assumption we have $\bar\theta\in N_{L_0/k}(L_0^*)$. So the injectivity of $w$ implies that $\theta=N_{L/F}(\xi_1)$ for some $\xi_1\in U_L$. Using the diagram \eqref{3p4p1} we find that $\xi.\bar\xi_1^{-1}\in L_0^*$ lies in $\ker(N_{L_0/k}: L_0^*\lra k^*)$. Now the surjectivity of $w'$ yields an element $z\in \ker(N_{L/F}: U_L\lra U_F)$ such that $\bar z=\xi_0.\bar\xi_1^{-1}$. Taking $\xi=z.\xi_1$, we get
  \[
  N_{L/F}(\xi)=N_{L/F}(\xi_1)N_{L/F}(z)=\theta\cdot 1=\theta
  \]and $\bar\xi=\bar z\cdot\bar\xi_1=\xi_0$ in $L_0$. This completes the proof.
\end{proof}

The proof of the lemma below is long. It will be postponed to the next section (see \eqref{4p3}).

\begin{lemma}\label{3p5}
  With notation and hypotheses as in $\eqref{3p1}$, suppose that $s=v_F(\lambda)=rp$ and that $k$ is a global function field.

  Then there exists a separable field extension $L_0/k$ of degree $p$ and an element $\xi_0\in L_0$ such that the following hold:

  \begin{enumerate}
    \item $N_{L_0/k}(\xi_0)=\bar\theta\in k$, and $r\bar\al'_{L_0}=(E_0L_0/L_0,\,\xi_0)\in\Br(L_0)$.
    \item If $\bar\al'_{E_0}=0$, then $L_0\subseteq E_0$ (up to $k$-isomorphism).
    \item If $\bar\al'_{E_0}\neq 0$ and $L_0\not\subseteq E_0$, then $\ind(\bar\al'_{E_0L_0})<\ind(\bar\al'_{E_0})$.
  \end{enumerate}
\end{lemma}

\newpara\label{3p6} {\bf Proof of Theorem\;\ref{1p1}}. Let us prove our main theorem assuming Lemma\;\ref{3p5}. Let $\al$ be the Brauer class of the division algebra $D$ in Theorem\;\ref{1p1}. As we have said in the introduction, it suffices to consider the case where $\al$ is tamely ramified, and of course we may assume $\al\neq 0$. The previous discussions in this section reduce the problem to the construction of an inductive pair in the case $s=v_F(\lambda)=rp$.

Let $L_0/k$ and $\xi_0\in L_0$ be chosen as in Lemma\;\ref{3p5}. Let $L/F$ be the unramified extension with residue field $L_0/k$. By Lemma\;\ref{3p4}, there exists $\xi\in U_L$ such that $\bar\xi=\xi_0\in L_0$, $N_{L/F}(\xi)=\theta$ and the element $\mu:=\xi.(-\pi)^r$ satisfies
\[
\al_L\cup(\mu)=0\in H^3(L)\quad\text{and}\quad N_{L/F}(\mu)=\theta.(-\pi)^{rp}=\lambda\,.
\]It remains to show $\ind(\al_L)<\ind(\al)$.

We distinguish three cases.

\noindent {\bf Case 1.} $\bar\al'_{E_0}=0$.

In this case we have $L_0\subseteq E_)$ by condition (2) in Lemma\;\ref{3p5}. Thus, $L\subseteq E$. Moreover, $\bar\al'_{E_0}$ implies $\al'_E=0$. Hence $\al'=(E/F,\,u)$ for some $u\in E^*$ and since $\al'$ is unramified, we have $u\in U_E$. Now $\al=\al'+(E/F,\,\pi)=(E/F,\,u\pi)$. Therefore,
\[
\al_L=(E/F,\,u\pi)_L=(EL/L\,,\,u\pi)=(E/L\,,\,u\pi)\,.
\]It follows that
\[
\ind(\al_L)=[E:L]<[E:F]=\ind(\al)\,.
\]

\noindent {\bf Case 2.} $\bar\al'_{E_0}\neq 0$ and $L_0\not\subseteq E_0$.

In this case we have $E\cap L=F$ and $[EL:L]=[E:F]$, because $[L:F]$ is a prime and $L\not\subseteq E$. Applying the index formula in \cite[Thm.\;5.15 (a)]{JW90} to the decompositions of $\al$ and of $\al_L=\al'_L+(EL/L,\,\pi)$, we obtain
\[
\begin{split}
  \ind(\al_L)&=\ind(\al'_{EL})[EL:L]=\ind(\bar\al'_{E_0L_0})[E:F]\\
  &<\ind(\bar\al'_{E_0})[E:F]=\ind(\al)\,,
\end{split}
\]the inequality here following from condition (3) of \eqref{3p5}.

\noindent {\bf Case 3.} $\bar\al'_{E_0}\neq 0$ and $L_0\subseteq E_0$.

Now we have $L\subseteq E$ and $\al_L=\al'_L+(E/L,\,\pi)$. Using the index formula \cite[Thm.\;5.15 (a)]{JW90} once again, we can deduce
\[
\ind(\al_L)=\ind(\al'_E)[E:L]<\ind(\al'_E)[E:F]=\ind(\al)\;.
\]So, the inequality $\ind(\al_L)<\ind(\al)$ holds in all cases. This shows that $(L,\,\mu)$ is an inductive pair, and the theorem is thus proved.

\section{Proofs of technical lemmas}\label{sec4}

Our goal in this section is to prove Lemma\;\ref{3p5}.

 We begin with the following observation.

\begin{lemma}\label{4p1}
  Let $K$ be a field of characteristic $p>0$ and $\gamma\in K^*$. Suppose that $p\nmid v(\gamma)$ for some normalized discrete valuation $v$ on $K$.

  Then there exists a cyclic extension $L/K$ of degree $p$ such that $\gamma\in N_{L/K}(L^*)$.
\end{lemma}
\begin{proof}
  Choose $c,\,d\in\Z$ such that $pc+dv(\gamma)=1$ and put $\rho=\gamma^d\varpi^{pc}$, where $\varpi\in K$ is a uniformizer for $v$. Then $\rho$ is another uniformizer for $v$. If $L/K$ is a cyclic degree $p$ extension such that $\rho\in N_{L/K}(L^*)$, then $\rho^d=\rho\varpi^{-pc}\in N_{L/K}(L^*)$ and this implies $\gamma\in N_{L/K}(L^*)$, since the group $K^*/N_{L/K}(L^*)$ is $p$-torsion and $p\nmid d$. Therefore, by replacing $\gamma$ with $\rho$ if necessary, we may assume $v(\gamma)=1$.

  We shall construct $L/K$ as an Artin--Schrier extension $L=K[X]/(X^p-X-\al)$, where $\al\in K$ and $\al\notin\wp(K):=\{a^p-a\,|\,a\in K\}$. We denote by $[\al,\,\lambda)$ the Brauer class of the cyclic $p$-algebra generated by two symbols $x,\,y$ over $K$ subject to the relations
  \[
  x^p-x=\al\,,\;y^p=\lambda\,,\;xy=y(x+1)\,.
  \]Then we have
  \begin{equation}\label{4p1p1}
    \begin{split}
    \lambda\in N_{L/K}(L^*)&\iff [\al,\,\lambda)=0\;\in\;\Br(K)\\
    &\iff \text{the polynomial equation } X^p_0-X_0-\al+\sum^{p-1}_{i=1}X_i^p\lambda^i=0\\
    &\;\;\;\text{ has a solution } (X_0,\dotsc, X_{p-1})\in K\times\cdots\times K\,
    \end{split}
  \end{equation}(for the second equivalence see e.g. \cite[p.234, Prop.\;2 (2)]{KatoKuzumaki} or \cite{ArasonBaeza2010}). We choose
  \[
  \al=\lambda^{-p}(\lambda+\lambda^2+\cdots+\lambda^{p-1})=\lambda^{1-p}(1+\lambda+\cdots+\lambda^{p-2})\,.
  \]
  Then the last condition in \eqref{4p1p1} holds for
  \[
  (X_0,\,X_1,\dotsc, X_{p-1})=(0,\,\lambda^{-1},\cdots,\lambda^{-1})\,.
  \]Moreover, we have $v(\al)=1-p<0$ and $p\nmid v(\al)$. This implies that $X^p-X-\al=0$ has no solution in $K$. Hence $\al\notin \wp(K)$ and
  $L:=K[X]/(X^p-X-\al)$ is a cyclic extension with the required property.
\end{proof}

Now we prove a characteristic $p$ version of \cite[Lemma\;3.1]{ParimalaPreetiSuresh2018}.

\begin{lemma}\label{4p2}
  Suppose given the following data:

  \begin{itemize}
    \item[$(\mathrm{i})$] A global field $k$ of characteristic $p>0$ and a cyclic extension $K/k$ of $p$-power degree.
    \item[$(\mathrm{ii})$]  A Brauer class $\beta\in \Br(k)[p^\infty]$, an integer $r\in\Z$ and an element $\theta\in k^*$ such that
    \[
    rp\beta=(K/k,\,\theta)\,.
    \]
    \item[$(\mathrm{iii})$] An integer $d\ge 2$.
      \item[$(\mathrm{iv})$] A finite set $S$ of places of $k$ including all the places $v$ such that $\beta\otimes k_v\neq 0\in \Br(k_v)$.
        \item[$(\mathrm{v})$] For each $v\in S$, a (connected or split) $(\Z/p\Z)$-Galois cover $L_v$ of $k_v$ and an element $\xi_v\in L_v^*$ such that
       \begin{enumerate}
         \item[(a)] $N_{L_v/k_v}(\xi_v)=\theta$;
         \item[(b)] $r\beta\otimes L_v=(KL_v/L_v,\,\xi_v)\in \Br(L_v)$;
         \item[(c)] $\ind(\beta\otimes KL_v)<d$.
       \end{enumerate}
    Here, by a $(\Z/p\Z)$-Galois cover $L_v/k_v$ we mean that $L_v$ is either a cyclic field extension of degree $p$ over $k_v$ or the direct product of $p$ copies of $k_v$. In the latter case, we identity $\Br(L_v)=\prod_{1\le i\le p}\Br(k_v)$ and let $(KL_v/L_v,\,\xi_v)$ denote
    $(Kk_v/k_v,\,\xi_{v,\,i})\in \prod_{1\le i\le p}\Br(k_v)$, where $(\xi_{v,\,1},\dotsc, \xi_{v,\,p})=\xi_v\in L_v=\prod_{1\le i\le p}k_v$.
  \end{itemize}

Then, there exists a separable field extension $L/k$ of degree $p$ and an element $\xi\in L^*$ such that

\begin{enumerate}
  \item $N_{L/k}(\xi)=\theta$;
  \item $r\beta\otimes L=(KL/L\,,\,\xi)\in \Br(L)$;
  \item $\ind(\al_L)<d$;
  \item $L\otimes_kk_v\cong L_v$ for all $v\in S$;
  \item $\xi$ is sufficiently close to $\xi_v$ in $L_v$ in the $v$-adic topology, for all $v\in S$.
\end{enumerate}
\end{lemma}
\begin{proof}
  First of all, we may further assume $S$ contains the following sets of places:

  \begin{itemize}
    \item $S_0:=\{v\,|\,v(\theta)\neq 0\}$;
    \item $S_1:=\{v\,|\, K/k \text{ is ramified at } v\}$;
    \item at least one place $v$ such that $L_v/k_v$ is a (cyclic) field extension.

    In fact, for any $v\notin S$, we have $\beta\otimes k_v=0$ by assumption (iv), whence $\ind(\beta\otimes KL_v)=\ind(\beta\otimes Kk_v)=1<2\le d$. On the other hand, Lemma\;\ref{4p1} tells us that there is a cyclic field extension $L_v/k_v$ of degree $p$ such that $\theta\in N_{L_v/k_v}(L_v^*)$. We choose any $\xi_v\in L_v^*$ such that $\theta=N_{L_v/k_v}(\xi_v)$. Then
    \[
    \mathrm{Cor}_{L_v/k_v}\left((KL_v/L_v,\,\xi_v)\right)=(Kk_v/k_v\,,\,N_{L_v/k_v}(\xi_v))=(Kk_v/k_v,\,\theta)=rp\beta\otimes k_v=0
    \]by assumption (ii). Since the corestriction map $\mathrm{Cor}: \Br(L_v)\to \Br(k_v)$ is injective (in fact an isomorphism) by local class field theory, we have $(KL_v/L_v,\,\xi_v)=0=r\beta\otimes L_v$. This means that condition (v) can still be satisfied when we add a new place $v$ for which $L_v$ is a field.
  \end{itemize}

  Now consider an arbitrary $v\in S$. By \cite[Lemma\;2.2]{ParimalaPreetiSuresh2018}, there exists $\rho_v\in L_v$ sufficiently close to $1$ such that $N_{L_v/k_v}(\rho_v)=1$ and $L_v=k_v(\xi_v\rho_v)$. We may assume $\rho_v$ is a norm from $KL_v/L_v$, by the fact that for any finite abelian extension $M/H$ of local fields, the norm subgroup $N_{M/H}(M^*)$ is open in $H^*$, even in the characteristic $p$ case (see e.g. \cite[p.219, $\S$XIV.6, Cor.\;1]{SerreLocalFields}).  Thus, it follows that $r\beta\otimes L_v=(KL_v/L_v,\,\xi_v)=(KL_v/L_v,\,\rho_v\xi_v)$. Replacing $\xi_v$ with $\xi_v\rho_v$ if necessary, we may assume $L_v=k_v(\xi_v)$.

  As in the proof of \cite[Lemma\;3.1]{ParimalaPreetiSuresh2018}, by using Chebotarev's density theorem and the strong approximation theorem, we can choose a place $v_0$ outsider $S$ such that $K/k$ splits completely at $v_0$ and we can find a monic polynomial
  \[
  f(X)=X^p+b_{p-1}X^{p-1}+\cdots+(-1)^p\theta
  \]with constant coefficient $f(0)=(-1)^p\theta$, which is close to the minimal polynomial of $\xi_v$ over $k_v$ for each $v\in S$. Then
  $L:=k[X]/(f)$ is a field extension of $k$, since we assumed that $L\otimes_kk_v=L_v$ is a field for at least one $v\in S$. The canonical image $\xi$ of $X$ in $L$ is an element with $N_{L/k}(\xi)=\theta$.

  The work of checking that the pair $(L,\,\xi)$ has all the required properties can be done in the same way as in \cite[p.416]{ParimalaPreetiSuresh2018}. The following fact will be used once again in this verification: For any finite abelian extension of local fields, any element close to 1 is a norm.
\end{proof}

\newpara\label{4p3} {\bf Proof of Lemma\;\ref{3p5}}. Now we can prove Lemma\;\ref{3p5}, which is essential in our proof of Theorem\;\ref{1p1}. Only three cases require consideration.

Case 1. $r\bar\al'_{E_0}=0$ and $E_0\neq k$.

Case 2. $r\bar\al'_{E_0}=0$ and $E_0=k$.

Case 3. $r\bar\al'_{E_0}\neq 0$.

In Case 1, we choose $L_0/k$ to be the unique degree $p$ subextension of the cyclic extension $E_0/k$. So there is no need to check \eqref{3p5} (3) and condition (2) of the lemma holds by the choice of $L_0$. It remains to show that one can find an element $\xi_0\in L_0$ satisfying
\[
N_{L_0/k}(\xi_0)=\bar\theta\in k\quad\text{and}\quad r\bar\al'_{L_0}=(E_0L_0/L_0,\,\xi_0)=(E_0/L_0,\,\xi_0)\,.
\]Indeed, from the assumption $r\bar\al'_{E_0}=0$ we know that
\[
r\bar\al'=(E_0/k,\,x)\quad\text{for some } x\in k^*\,.
\]Hence $(E_0/k,\,x^p)=rp\bar\al'=(E_0/k,\,\theta)$, which implies $\bar\theta x^{-p}=N_{E_0/k}(y)$ for some $y\in E_0$. Let $\xi_0=N_{E_0/L_0}(y)x$. Then
\[
N_{L_0/k}(\xi_0)=N_{E_0/k}(y)x^p=\bar x^{-p}.x^p=\bar\theta\,,
\]and
\[
\begin{split}
  r\bar\al'_{L_0}&=(E_0/k,\,x)_{L_0}=(E_0/L_0,\,x)=(E_0/L_0,\,\xi_0N_{E_0/L_0}(y)^{-1})\\
  &=(E_0/L_0\,,\,\xi_0)\;\in\;\Br(L_0)\,.
\end{split}
\]This completes the proof in Case 1.

Now consider Case 2, where $E_0=k$ and $r\bar\al'=r\al'_{E_0}=0$.

In this case we have $\al'=\al\neq 0$ and $\bar\al'_{E_0}=\bar\al'\neq 0$. So \eqref{3p5} (2) is needless to check. To prove the desired result, we shall apply Lemma\;\ref{4p2} with
\[
\beta=\bar\al'\,,\; K=E_0\quad\text{and}\quad d=\ind(\bar\al'_{E_0})
\]and we will take the pair $(L,\,\xi)$ obtained in \eqref{4p2} to be the pair $(L_0,\,\xi_0)$ asserted in \eqref{3p5}.

Note that conditions (1) in \eqref{3p5} is guaranteed by assertions (1) and (2) in Lemma\;\ref{4p2}, and \eqref{3p5} (3) is part of \eqref{4p2} (3). The only remaining job is to take $S$ to be the finite set of places $v$ such that $\bar\al'\otimes k_v\neq 0$ and to provide for each $v\in S$ a pair $(L_v,\,\xi_v)$ satisfying assumption (v) in \eqref{4p2}.

If $p\nmid v(\bar\theta)$, we can find a cyclic extension $L_v/k_v$ of degree $p$ and $\xi_v\in L_v^*$ such that $\bar\theta=N_{L_v/k_v}(\xi_v)$, whence condition (a) in \eqref{4p2}  (v). Our assumption in the present case says $E_0=k$ and $r\beta=r\bar\al'_{E_0}=0$. So condition (b) in \eqref{4p2} (v) holds as well. For condition (c), it suffices to observe that
\[
\begin{split}
  \ind(\beta\otimes KL_v)&=\ind(\bar\al'\otimes E_0L_v)=\ind(\bar\al'\otimes L_v)\\
  &=\frac{\ind(\bar\al'\otimes k_v)}{[L_v:k_v]}<\ind(\bar\al'\otimes k_v)\le \ind(\bar\al')=d\,,
\end{split}
\]
since $k_v$ is a local field.

Next assume $p\,|\,v(\bar\theta)$. We now choose $L_v/k_v$ to be the unique unramified extension of degree $p$. Then $\bar\theta$ is a norm from $L_v/k_v$. Conditions (b) and (c) can be checked as in the previous paragraph.

Case 2 is thus solved.

In the third case, $r\bar\al'_{E_0}\neq 0$. As in Case 2, we will use \eqref{4p2} by taking $\beta=\bar\al'$, $K=E_0$ and $d=\ind(\bar\al'_{E_0})$. Again, it is sufficient to construct, for each place $v$ such that $\bar\al'\otimes k_v\neq 0$, a pair $(L_v,\,\xi_v)$ satisfying \eqref{4p2} (v).

Indeed, as in Case 2 we can find a cyclic extension $L_v/k_v$ and an element $\xi_v\in L_v$ satisfying condition (a) in \eqref{4p2} (v). Condition (b) can be deduced from (a) by using the corestriction map $\mathrm{Cor}: \Br(L_v)\to \Br(k_v)$ and its injectivity. So it remains to check condition (c) of \eqref{4p2} (v).

Note that $E_0k_v\cap L_v$ is either $k_v$ or $L_v$.

If $E_0k_v\cap L_v=k_v$, we have $[E_0L_v:E_0k_v]=p=[L_v:k_v]$ and
\[
\ind(\bar\al'\otimes E_0L_v)<
\begin{cases}
2\le d=\ind(\bar\al'_{E_0})\,, & \text{ if } \bar\al'\otimes E_0k_v=0\\
\ind(\bar\al'\otimes E_0k_v)\le d=\ind(\bar\al'_{E_0})\,, & \text{ if }\bar\al'\otimes E_0k_v\neq 0
\end{cases}
\] since $E_0k_v$ is a local field.

If $E_0k_v\cap L_v=L_v$, then $E_0k_v=E_0L_v$. Writing $rp=mp^t$ with $m$ coprime to $p$ and $t\ge 1$, we have
\begin{equation}\label{4p3p1}
  \begin{split}
    \ind(\bar\al'\otimes E_0L_v)&=\ind(\bar\al'\otimes E_0k_v)=\per(\bar\al'\otimes E_0k_v)=\frac{\per(\bar\al'\otimes k_v)}{[E_0k_v:k_v]}\\
    &\le p^t\frac{\per(p^t\bar\al'\otimes k_v)}{[E_0k_v:k_v]}=p^t\frac{\per(rp\bar\al'\otimes k_v)}{[E_0k_v:k_v]}\\
    &=p^t\frac{\per\left((E_0/k,\,\bar\theta)\otimes k_v\right)}{[E_0k_v:k_v]}=p^t\frac{\per\left((E_0k_v/k_v,\,\bar\theta)\right)}{[E_0k_v:k_v]}\\
    &=p^t\frac{\per\left((E_0L_v/k_v,\,\bar\theta)\right)}{[E_0L_v:k_v]}
  \end{split}
\end{equation}
Since
\[
\bar\theta^{[E_0L_v: L_v]}=N_{L_v/k_v}(\xi_v)^{[E_oL_v:L_v]}=N_{E_0L_v/k_v}(\xi_v)
\]we have
\[
\per\left((E_0L_v/k_v,\,\bar\theta)\right)\le [E_0L_v: L_v]\,.
\]Therefore, from \eqref{4p3p1} we deduce that
\[
\begin{split}
  \ind(\bar\al'\otimes E_0L_v)&\le   p^t\frac{\per\left((E_0L_v/k_v,\,\bar\theta)\right)}{[E_0L_v:k_v]}\\
     &\le p^t\frac{[E_0L_v: L_v]}{[E_0L_v:k_v]}=p^{t-1}<p^t\,.
\end{split}
\]
To finish the proof, it remains to check that $d=\ind(\bar\al'_{E_0})=p^t$. Indeed, from the fact that $mp^t\bar\al'=rp\bar\al'=(E_0/k,\,\bar\theta)$ we see that $d=\ind(\bar\al'_{E_0})=\per(\bar\al'_{E_0})$ divides $p^t$ (noticing that $E_0$ is a global field). On the other hand, our hypothesis in the present case (Case 3) is that $mp^{t-1}\bar\al'_{E_0}=r\bar\al'_{E_0}\neq 0$. Hence $d=\per(\bar\al'_{E_0})>p^{t-1}$. So we have $d=p^t$ as desired.

This completes the proof of Lemma\;\ref{3p5}.

\

\

\noindent \emph{Acknowledgements.} This work  is supported by a grant from the National Natural Science Foundation of China (Project No. 11801260). The author thanks Zhengyao Wu for helpful discussions.

\addcontentsline{toc}{section}{\textbf{References}}

\bibliographystyle{alpha}

\bibliography{RedNorm}

\

Contact information of the author:

\

Yong HU

\medskip

Department of Mathematics

Southern University of Science and Technology

No. 1088, Xueyuan Blvd., Nanshan district

518055 Shenzhen, Guangdong,

P.R. China

Email: huy@sustc.edu.cn

\

\clearpage \thispagestyle{empty}

\end{document}